\documentclass[preprint,3p,11pt,numbers]{elsarticle}

\makeatletter
\def\ps@pprintTitle{%
 \let\@oddhead\@empty
 \let\@evenhead\@empty
 \def\@oddfoot{\centerline{\thepage}}%
 \let\@evenfoot\@oddfoot}
\makeatother
\usepackage{amssymb,amsthm,mathtools}
\mathtoolsset{showonlyrefs=true} % Only numbers equations that are referenced

\usepackage{dsfont} % to get \mathds for indicator functions
\usepackage{geometry}
\usepackage[colorlinks]{hyperref}

\newtheorem{Theorem}{Theorem}
\newtheorem{Lemma}{Lemma}

\newcommand{\N}{\mathbb{N}}
\newcommand{\R}{\mathbb{R}}

\newcommand{\EE}{\mathbb{E}}
\newcommand{\bb}[1]{\boldsymbol{#1}}
\newcommand{\leqdef}{\vcentcolon=}
\DeclareRobustCommand{\brkbinom}{\genfrac\{\}{0pt}{}}

\allowdisplaybreaks

\begin{document}

\begin{frontmatter}

    \title{Moments of the negative multinomial distribution}%

    \author[a1]{Fr\'ed\'eric Ouimet\texorpdfstring{}{)}}%

    \address[a1]{Universit\'e de Montr\'eal, Montreal, QC H3T 1J4, Canada.}%

    \cortext[cor1]{Corresponding author}%
    \ead{frederic.ouimet@umontreal.ca}%

    \begin{abstract}
        The negative multinomial distribution appears in many areas of applications such as polarimetric image processing and the analysis of longitudinal count data. In previous studies, \cite{doi:10.2307/2333745} derived general formulas for the falling factorial moments of the negative multinomial distribution, while \cite{MR3168930} obtained expressions for the cumulants. Despite the availability of the moment generating function, no comprehensive formulas for the moments have been calculated thus far. This paper addresses this gap by presenting general formulas for both central and non-central moments of the negative multinomial distribution. These formulas are expressed in terms of binomial coefficients and Stirling numbers of the second kind. Utilizing these formulas, we provide explicit expressions for all central moments up to the $4^{\text{th}}$~order and all non-central moments up to the $8^{\text{th}}$~order.
    \end{abstract}


\end{frontmatter}

\section{Introduction}\label{sec:intro}

    The negative multinomial distribution is a probability distribution that can be used to model count data, where the outcome of interest is the number of occurrences of $d\in \N$ different events when the number of failures (a failure means that, for a given trial, an object is not categorized in any of the $d$ categories) is a fixed value $r\in \N$. It is a multivariate generalization of the well-known negative binomial distribution, for which $d = 1$.
    For a general reference on the negative multinomial distribution and its properties, refer to \citet{MR171341} or Chapter~36 of \citet{MR1429617}.

    One of the main motivations for using the negative multinomial distribution is its ability to model overdispersion for count vectors, which happens when the variances of the count variables are larger than their mean (\citet{MR2063401}). The Poisson distribution for instance assumes that the mean and variance are equal but, in many real-world scenarios, this is often not the case. The negative multinomial distribution allows for modeling overdispersion by allowing for different variances for each event type (\citet{MR3155491}).

    Another motivation for using the negative multinomial distribution is its ability to handle excess zeros in count data, see, e.g., \citet{MR4481433}. Count data often exhibit zero inflation, where there are more zeros than would be expected under a Poisson distribution. The negative multinomial distribution provides a flexible framework for modeling zero inflation by allowing for different probabilities of zero occurrences for each event type.

    A third motivation for using the negative multinomial distribution is its ability to model count data with multiple event types, see, e.g., \cite{MR801857,doi:10.1080/03610928508829007,doi:10.2307/271020,doi:10.1080/01621459.1999.10474178,MR1681133,MR1475055,MR3610402,MR4242807}. In many applications, there is more than one type of event that can occur and the negative multinomial distribution allows for modeling the counts of each event type simultaneously. This is particularly useful in fields such as marketing, where the goal is to model the number of purchases of different products, or in ecology, where the goal is to model the counts of different species in a community \cite{doi:10.3390/f11050571}.

    Overall, the negative multinomial distribution provides a flexible and powerful tool for modeling count data in a variety of applications such as polarimetric image processing~\cite{MR3037959}, the analysis of RNA-seq.\ data \cite{Kusi-Appiah2016phd}, pollen analysis \cite{doi:10.2307/2333745}, longitudinal data \cite{MR1475055,doi:10.1080/01621459.1999.10474178,MR1681133}, etc. Its theoretical properties have been investigated in numerous papers, see, e.g., \cite{MR207079,MR345311,MR372951,MR408075,MR656170,MR813445,MR1058927,MR1219279,MR2047690,MR2255791,MR2370884,MR3197132,MR3263054,MR3825458,MR4122073,MR4505544}. One can find extensions of the model in \cite{MR208734,doi:10.1080/03610929508831473,Kusi-Appiah2016phd,MR4358040}.

    The ability of the negative multinomial distribution to handle overdispersion, zero inflation, and multiple event types makes it a valuable tool for data scientists and statisticians. Whether one is interested in modeling the number of purchases of different products, the counts of different species in a community, or any other count data, the negative multinomial distribution can provide valuable insights and inform decision-making.

    In previous studies, \citet{doi:10.2307/2333745} derived general formulas for the falling factorial moments of the negative multinomial distribution, while \citet{MR3168930} obtained expressions for the cumulants. Despite the availability of the moment generating function, no comprehensive formulas for the moments have been calculated thus far. Our goal in this paper is to address this gap by presenting general formulas for both central and non-central moments of the negative multinomial distribution.

    Here is an outline of the paper.
    In Section~\ref{sec:preliminaries}, the necessary definitions and notations are introduced, along with a preliminary result on factorial moments of the negative multinomial distribution due to \citet{doi:10.2307/2333745}.
    The general formulas for the central and non-central moments of the negative multinomial distribution are stated and proved in Section~\ref{sec:main.results}.
    The numerical implementation of those formulas in \texttt{Mathematica} is provided in Section~\ref{sec:numerical.codes}.
    Finally, in Section~\ref{sec:explicit.formulas}, our general formulas are applied to give explicit expressions for all central moments up to the fourth~order and all non-central moments up to the eighth~order. Open problems of interest are stated in Section~\ref{sec:open.problems}.

    \section{The Negative Multinomial Distribution}\label{sec:preliminaries}

    For any $d\in \N$, let $\bb{x}\in [0,1]^d$ be such that $\|\bb{x}\|_1 \leqdef \sum_{i=1}^d |x_i| < 1$.
    The probability mass function $\bb{k}\mapsto P_{r,\bb{x}}(\bb{k})$ of the negative multinomial distribution is defined by
    \begin{align}\label{eq:multinomial.probability}
        P_{r,\bb{x}}(\bb{k})
        &\leqdef \frac{\Gamma(r + \|\bb{k}\|_1)}{\Gamma(r) \prod_{i=1}^d \Gamma(k_i + 1)} (1 - \|\bb{x}\|_1)^r \prod_{i=1}^d x_i^{k_i} \notag \\
        &= \frac{\Gamma(r + \|\bb{k}\|_1)}{\Gamma(r) \prod_{i=1}^d \Gamma(k_i + 1)} (1 - \|\bb{x}\|_1)^{r + \|\bb{k}\|_1} \prod_{i=1}^d y_i^{k_i}, \quad \bb{k}\in \N_0^d,
    \end{align}
    where $r > 0$ is a positive real number and $y_i \leqdef x_i / (1 - \|\bb{x}\|_1)$ for all $i\in \{1,\dots,d\}$.
    If a random vector $\bb{\eta} = (\eta_1,\dots,\eta_d)$ follows this distribution, we write for short $\bb{\eta}\sim \mathrm{Neg\hspace{0.3mm}Multinomial}\hspace{0.2mm}(r,\bb{x})$.
    In this paper, our main goal is to give general formulas for the non-central and central moments of \eqref{eq:multinomial.probability}, namely
    \begin{equation}\label{eq:main.goal}
        \EE\left[\prod_{i=1}^d \eta_i^{p_i}\right] \quad \text{and} \quad \EE\left[\prod_{i=1}^d (\eta_i - \EE[\eta_i])^{p_i}\right], \qquad p_1,\dots,p_d\in \N_0.
    \end{equation}

    We obtain the formulas using a combinatorial argument and the general expression for the falling factorial moments found by \citet{doi:10.2307/2333745}, which we register in the lemma below.
    \begin{Lemma}[Factorial moments]\label{lem:factorial.moments}
        Let $\bb{\eta}\sim \mathrm{Neg\hspace{0.3mm}Multinomial}\hspace{0.3mm}(r,\bb{x})$.
        Then, for all $k_1,\dots,k_d\in \N_0$,
        \begin{equation*}\label{eq:factorial.moments}
            \EE\left[\prod_{i=1}^d \eta_i^{(k_i)}\right] = \left(r - 1 + \|\bb{k}\|_1\right)^{(\|\bb{k}\|_1)} \prod_{i=1}^d y_i^{k_i},
        \end{equation*}
        where $m^{(k)} \leqdef m (m-1) \dots (m-k+1)$ denotes the {\it $k^{\text{th}}$~order~falling factorial of $m$}.
    \end{Lemma}
    The formulas we derive for the expectations in Equation \eqref{eq:main.goal} will be employed to calculate all the central moments up to the fourth order, as well as all the non-central moments up to the eighth order.
    For information about the moment generating function, the cumulant generating function, and expressions for the cumulants, refer to \citet{MR3168930}.

\section{Results}\label{sec:main.results}

    First, we give a general formula for the non-central moments of the negative multinomial distribution in \eqref{eq:multinomial.probability}.

    \begin{Theorem}[Non-central moments]\label{thm:non.central.moments}
        Let $\bb{\eta}\sim \mathrm{Neg\hspace{0.3mm}Multinomial}\hspace{0.3mm}(r,\bb{x})$.
        Then, for all \mbox{$p_1,\dots,p_d\in \N_0$,}
        \begin{equation*}
            \EE\left[\prod_{i=1}^d \eta_i^{p_i}\right] = \sum_{k_1=0}^{p_1} \dots \sum_{k_d=0}^{p_d} \left(r - 1 + \|\bb{k}\|_1\right)^{(\|\bb{k}\|_1)} \prod_{i=1}^d \brkbinom{p_i}{k_i} y_i^{k_i},
        \end{equation*}
        where $\brkbinom{p}{k}$ denotes a Stirling number of the second kind (i.e., the number of ways to partition a set of $p$ objects into $k$ non-empty subsets); recall that
        \begin{equation*}
            y_i \leqdef \frac{x_i}{1 - \|\bb{x}\|_1}, \quad \text{for all } i\in \{1,\dots,d\}.
        \end{equation*}
    \end{Theorem}

    \begin{proof}
        The following well-known relationship between the power $p\in \N_0$ of a number $x\in \R$ and the falling factorials of $x$ is already established:
        \begin{equation*}
            x^p = \sum_{k=0}^p \brkbinom{p}{k} \, x^{(k)}.
        \end{equation*}
        This relationship can be found in \cite{MR1397498} (p. 262). By applying this formula to each $\eta_{i}^{p_{i}}$ and utilizing the linearity of expectation, we obtain the following:
        \begin{equation*}
            \EE\left[\prod_{i=1}^d \eta_i^{p_i}\right] = \sum_{k_1=0}^{p_1} \dots \sum_{k_d=0}^{p_d} \brkbinom{p_1}{k_1} \dots \brkbinom{p_d}{k_d} \, \EE\left[\prod_{i=1}^d \eta_i^{(k_i)}\right],
        \end{equation*}
        Therefore, the conclusion is a direct consequence of Lemma~\ref{lem:factorial.moments}.
    \end{proof}

    We can now derive a comprehensive formula for the central moments of the negative multinomial distribution.

    \begin{Theorem}[Central moments]\label{thm:central.moments}
        Let $\bb{\eta}\sim \mathrm{Neg\hspace{0.3mm}Multinomial}\hspace{0.3mm}(r,\bb{x})$.
        Then, for all $p_1,\dots,p_d\in \N_0$,
        \begin{equation*}
            \begin{aligned}
                \EE\left[\prod_{i=1}^d (\eta_i - \EE[\eta_i])^{p_i}\right]
                &= \sum_{\ell_1=0}^{p_1} \dots \sum_{\ell_d=0}^{p_d} \sum_{k_1=0}^{\ell_1} \dots \sum_{k_d=0}^{\ell_d} \left(r - 1 + \|\bb{k}\|_1\right)^{(\|\bb{k}\|_1)} \\
                &\quad\cdot (-r)^{\sum_{i=1}^d (p_i - \ell_i)} \prod_{i=1}^d \binom{p_i}{\ell_i} \brkbinom{\ell_i}{k_i} y_i^{p_i - \ell_i + k_i},
            \end{aligned}
        \end{equation*}
        where $\binom{p}{\ell}$ denotes the binomial coefficient $\frac{p!}{\ell! (p - \ell)!}$; recall that
        \begin{equation*}
            y_i \leqdef \frac{x_i}{1 - \|\bb{x}\|_1}, \quad \text{for all } i\in \{1,\dots,d\}.
        \end{equation*}
    \end{Theorem}

    \begin{proof}
        By applying the binomial formula to each factor $(\eta_i - \EE[\eta_i])^{p_i}$ and using the fact that $\EE[\eta_i] = r y_i$ for all $i\in \{1,\dots,d\}$, note that
        \begin{equation*}
            \EE\left[\prod_{i=1}^d (\eta_i - \EE[\eta_i])^{p_i}\right] = \sum_{\ell_1=0}^{p_1} \dots \sum_{\ell_d=0}^{p_d} \EE\left[\prod_{i=1}^d \eta_i^{\ell_i}\right] \cdot \prod_{i=1}^d \binom{p_i}{\ell_i} (-r y_i)^{p_i - \ell_i}.
        \end{equation*}
        Therefore, the conclusion is a direct consequence of Theorem~\ref{thm:non.central.moments}.
    \end{proof}

\section{Numerical Codes}\label{sec:numerical.codes}

    The formulas in Theorems~\ref{thm:non.central.moments} and \ref{thm:central.moments} can be put into practice in \texttt{{Mathematica}} through the following procedure:
    \begin{verbatim}
    NonCentral[r_, x_, p_, d_] :=
      Sum[FactorialPower[r - 1 + Sum[k[i], {i, 1, d}],
      Sum[k[i], {i, 1, d}]] * Product[StirlingS2[p[[i]], k[i]] *
      (x[[i]] / (1 - Sum[x[[i]], {i, 1, d}])) ^ k[i], {i, 1, d}], ##] & @@
      ({k[#], 0, p[[#]]} & /@ Range[d]);
    Central[r_, x_, p_, d_] :=
      Sum[Sum[FactorialPower[r - 1 + Sum[k[i], {i, 1, d}],
      Sum[k[i], {i, 1, d}]] * (-r) ^ Sum[p[[i]] - ell[i], {i, 1, d}]
      * Product[Binomial[p[[i]], ell[i]] * StirlingS2[ell[i], k[i]]
      * (x[[i]] / (1 - Sum[x[[i]], {i, 1, d}])) ^
      (p[[i]] - ell[i] + k[i]), {i, 1, d}], ##] & @@
      ({k[#], 0, ell[#]} & /@ Range[d]), ##] & @@
      ({ell[#], 0, p[[#]]} & /@ Range[d]);
    \end{verbatim}

\section{Explicit Formulas}\label{sec:explicit.formulas}

    In the two subsections below, we calculate (explicitly) all the non-central moments up to the eighth~order and all the central moments up to the fourth~order.
    Here is a table of the Stirling numbers of the second kind that we will use in our~calculations:
    \footnotesize
    \begin{equation*}
        \begin{aligned}
            &\brkbinom{0}{0} = 1, \\[0.5mm]
            &\brkbinom{1}{0} = 0, ~ \brkbinom{1}{1} = 1, \\[0.5mm]
            &\brkbinom{2}{0} = 0, ~ \brkbinom{2}{1} = 1, \brkbinom{2}{2} = 1, \\[0.5mm]
            &\brkbinom{3}{0} = 0, ~ \brkbinom{3}{1} = 1, \brkbinom{3}{2} = 3, ~ \brkbinom{3}{3} = 1, \\[0.5mm]
            &\brkbinom{4}{0} = 0, ~ \brkbinom{4}{1} = 1, \brkbinom{4}{2} = 7, ~ \brkbinom{4}{3} = 6, \brkbinom{4}{4} = 1, \\[0.5mm]
            &\brkbinom{5}{0} = 0, ~ \brkbinom{5}{1} = 1, \brkbinom{5}{2} = 15, ~ \brkbinom{5}{3} = 25, \brkbinom{5}{4} = 10, ~ \brkbinom{5}{5} = 1, \\[0.5mm]
            &\brkbinom{6}{0} = 0, ~ \brkbinom{6}{1} = 1, \brkbinom{6}{2} = 31, ~ \brkbinom{6}{3} = 90, \brkbinom{6}{4} = 65, ~ \brkbinom{6}{5} = 15, \brkbinom{6}{6} = 1, \\[0.5mm]
            &\brkbinom{7}{0} = 0, ~ \brkbinom{7}{1} = 1, \brkbinom{7}{2} = 63, ~ \brkbinom{7}{3} = 301, \brkbinom{7}{4} = 350, ~ \brkbinom{7}{5} = 140, \brkbinom{7}{6} = 21, ~ \brkbinom{7}{7} = 1, \\[0.5mm]
            &\brkbinom{8}{0} = 0, ~ \brkbinom{8}{1} = 1, \brkbinom{8}{2} = 127, ~ \brkbinom{8}{3} = 966, \brkbinom{8}{4} = 1701, ~ \brkbinom{8}{5} = 1050, \brkbinom{8}{6} = 266, ~ \brkbinom{8}{7} = 28, ~ \brkbinom{8}{8} = 1.
        \end{aligned}
    \end{equation*}
    \normalsize

    \vspace{1mm}
    \subsection{Computation of the Non-Central Moments up to the Eighth~Order}\label{sec:noncentral}

    By utilizing the general expression outlined in Theorem~\ref{thm:non.central.moments} and eliminating the Stirling numbers $\brkbinom{p_i}{k_i}$ that are equal to zero, we obtain the following results effortlessly.

    \vspace{3mm}
    \noindent
    %MDPI: Please confirm if the underline in text should be retained. (yes it should)
    \underline{1st order:} For $j_1\in \{1,\dots,d\}$,
    \begin{align*}
        \EE[\eta_{j_1}]
        &= y_{j_1} r.
    \end{align*}

    \vspace{2mm}
    \noindent
    \underline{2nd order:} For different $j_1, j_2\in \{1,\dots,d\}$,
    \begin{align*}
        \EE[\eta_{j_1}^2]
        &= y_{j_1} \big[r + (r + 1)^{(2)} y_{j_1}\big], \\[1.1mm]
        \EE[\eta_{j_1} \eta_{j_2}]
        &= y_{j_1} y_{j_2} (r + 1)^{(2)}.
    \end{align*}

    \vspace{2mm}
    \noindent
    \underline{3rd order:} For different $j_1, j_2, j_3\in \{1,\dots,d\}$,
    \begin{align*}
        \EE[\eta_{j_1}^3]
        &= y_{j_1} \big[r + 3 (r + 1)^{(2)} y_{j_1} + (r + 2)^{(3)} y_{j_1}^2\big], \\[1.1mm]
        \EE[\eta_{j_1}^2 \eta_{j_2}]
        &= y_{j_1} y_{j_2} \big[(r + 1)^{(2)} + (r + 2)^{(3)} y_{j_1}\big], \\[1.1mm]
        \EE[\eta_{j_1} \eta_{j_2} \eta_{j_3}]
        &= y_{j_1} y_{j_2} y_{j_3} (r + 2)^{(3)}.
    \end{align*}

    \vspace{2mm}
    \noindent
    \underline{4th order:} For different $j_1, j_2, j_3, j_4\in \{1,\dots,d\}$,
    \begin{align*}
        \EE[\eta_{j_1}^4]
        &= y_{j_1} \big[r + 7 (r + 1)^{(2)} y_{j_1} + 6 (r + 2)^{(3)} y_{j_1}^2 + (r + 3)^{(4)} y_{j_1}^3\big], \\[1.1mm]
        \EE[\eta_{j_1}^3 \eta_{j_2}]
        &= y_{j_1} y_{j_2} \big[(r + 1)^{(2)} + 3 (r + 2)^{(3)} y_{j_1} + (r + 3)^{(4)} y_{j_1}^2\big], \\[1.1mm]
        \EE[\eta_{j_1}^2 \eta_{j_2}^2]
        &= y_{j_1} y_{j_2} \big[(r + 1)^{(2)} + (r + 2)^{(3)} (y_{j_1} + y_{j_2}) + (r + 3)^{(4)} y_{j_1} y_{j_2}\big], \\[1.1mm]
        \EE[\eta_{j_1}^2 \eta_{j_2} \eta_{j_3}]
        &= y_{j_1} y_{j_2} y_{j_3} \big[(r + 2)^{(3)} + (r + 3)^{(4)} y_{j_1}\big], \\[1.1mm]
        \EE[\eta_{j_1} \eta_{j_2} \eta_{j_3} \eta_{j_4}]
        &= y_{j_1} y_{j_2} y_{j_3} y_{j_4} (r + 3)^{(4)}.
    \end{align*}

    \vspace{2mm}
    \noindent
    \underline{5th order:} For different $j_1, j_2, j_3, j_4, j_5\in \{1,\dots,d\}$,
    \begin{align*}
        \EE[\eta_{j_1}^5]
        &= y_{j_1} \left[\hspace{-1mm}
            \begin{array}{l}
                r + 15 (r + 1)^{(2)} y_{j_1} + 25 (r + 2)^{(3)} y_{j_1}^2 \\[0.49mm]
                + 10 (r + 3)^{(4)} y_{j_1}^3 + (r + 4)^{(5)} y_{j_1}^4
            \end{array}
            \hspace{-1mm}\right], \\[1.1mm]
        \EE[\eta_{j_1}^4 \eta_{j_2}]
        &= y_{j_1} y_{j_2} \left[\hspace{-1mm}
            \begin{array}{l}
                (r + 1)^{(2)} + 7 (r + 2)^{(3)} y_{j_1} \\[0.49mm]
                + 6 (r + 3)^{(4)} y_{j_1}^2 + (r + 4)^{(5)} y_{j_1}^3
            \end{array}
            \hspace{-1mm}\right], \\[1.1mm]
        \EE[\eta_{j_1}^3 \eta_{j_2}^2]
        &= y_{j_1} y_{j_2} \left[\hspace{-1mm}
            \begin{array}{l}
                (r + 1)^{(2)} + (r + 2)^{(3)} (3 y_{j_1} + y_{j_2}) \\[0.49mm]
                + (r + 3)^{(4)} (y_{j_1}^2 + 3 y_{j_1} y_{j_2}) + (r + 4)^{(5)} y_{j_1}^2 y_{j_2}
            \end{array}
            \hspace{-1mm}\right], \\[1.1mm]
        \EE[\eta_{j_1}^3 \eta_{j_2} \eta_{j_3}]
        &= y_{j_1} y_{j_2} y_{j_3} \big[(r + 2)^{(3)} + 3 (r + 3)^{(4)} y_{j_1} + (r + 4)^{(5)} y_{j_1}^2\big], \\[1.1mm]
        \EE[\eta_{j_1}^2 \eta_{j_2}^2 \eta_{j_3}]
        &= y_{j_1} y_{j_2} y_{j_3} \big[(r + 2)^{(3)} + (r + 3)^{(4)} (y_{j_1} + y_{j_2}) + (r + 4)^{(5)} y_{j_1} y_{j_2}\big], \\[1.1mm]
        \EE[\eta_{j_1}^2 \eta_{j_2} \eta_{j_3} \eta_{j_4}]
        &= y_{j_1} y_{j_2} y_{j_3} y_{j_4} \big[(r + 3)^{(4)} + (r + 4)^{(5)} y_{j_1}\big], \\[1.1mm]
        \EE[\eta_{j_1} \eta_{j_2} \eta_{j_3} \eta_{j_4} \eta_{j_5}]
        &= y_{j_1} y_{j_2} y_{j_3} y_{j_4} y_{j_5} (r + 4)^{(5)}.
    \end{align*}

    \vspace{2mm}
    \noindent
    \underline{6th order:} For different $j_1, j_2, j_3, j_4, j_5, j_6\in \{1,\dots,d\}$,
    \begin{align*}
        \EE[\eta_{j_1}^6]
        &= y_{j_1} \left[\hspace{-1mm}
            \begin{array}{l}
                r + 31 (r + 1)^{(2)} y_{j_1} + 90 (r + 2)^{(3)} y_{j_1}^2 \\[0.49mm]
                + 65 (r + 3)^{(4)} y_{j_1}^3 + 15 (r + 4)^{(5)} y_{j_1}^4 + (r + 5)^{(6)} y_{j_1}^5
            \end{array}
            \hspace{-1mm}\right], \\[1.1mm]
        \EE[\eta_{j_1}^5 \eta_{j_2}]
        &= y_{j_1} y_{j_2} \left[\hspace{-1mm}
            \begin{array}{l}
                (r + 1)^{(2)} + 15 (r + 2)^{(3)} y_{j_1} + 25 (r + 3)^{(4)} y_{j_1}^2 \\[0.49mm]
                + 10 (r + 4)^{(5)} y_{j_1}^3 + (r + 5)^{(6)} y_{j_1}^4
            \end{array}
            \hspace{-1mm}\right], \\[1.1mm]
        \EE[\eta_{j_1}^4 \eta_{j_2}^2]
        &= y_{j_1} y_{j_2} \left[\hspace{-1mm}
            \begin{array}{l}
                (r + 1)^{(2)} + (r + 2)^{(3)} (7 y_{j_1} + y_{j_2}) \\[0.49mm]
                + (r + 3)^{(4)} (6 y_{j_1}^2 + 7 y_{j_1} y_{j_2}) \\[0.49mm]
                + (r + 4)^{(5)} (y_{j_1}^3 + 6 y_{j_1}^2 y_{j_2}) + (r + 5)^{(6)} y_{j_1}^3 y_{j_2}
            \end{array}
            \hspace{-1mm}\right], \\[1.1mm]
        \EE[\eta_{j_1}^4 \eta_{j_2} \eta_{j_3}]
        &= y_{j_1} y_{j_2} y_{j_3} \left[\hspace{-1mm}
            \begin{array}{l}
                (r + 2)^{(3)} + 7 (r + 3)^{(4)} y_{j_1} \\[0.49mm]
                + 6 (r + 4)^{(5)} y_{j_1}^2 + (r + 5)^{(6)} y_{j_1}^3
            \end{array}
            \hspace{-1mm}\right], \\[1.1mm]
        \EE[\eta_{j_1}^3 \eta_{j_2}^3]
        &= y_{j_1} y_{j_2} \left[\hspace{-1mm}
            \begin{array}{l}
                (r + 1)^{(2)} + (r + 2)^{(3)} (3 y_{j_1} + 3 y_{j_2}) \\[0.49mm]
                + (r + 3)^{(4)} (y_{j_1}^2 + 9 y_{j_1} y_{j_2} + y_{j_2}^2) \\[0.49mm]
                + (r + 4)^{(5)} (3 y_{j_1}^2 y_{j_2} + 3 y_{j_1} y_{j_2}^2) + (r + 5)^{(6)} y_{j_1}^2 y_{j_2}^2
            \end{array}
            \hspace{-1mm}\right], \\[1.1mm]
        \EE[\eta_{j_1}^3 \eta_{j_2}^2 \eta_{j_3}]
        &= y_{j_1} y_{j_2} y_{j_3} \left[\hspace{-1mm}
            \begin{array}{l}
                (r + 2)^{(3)} + (r + 3)^{(4)} (3 y_{j_1} + y_{j_2}) \\[0.49mm]
                + (r + 4)^{(5)} (y_{j_1}^2 + 3 y_{j_1} y_{j_2}) + (r + 5)^{(6)} y_{j_1}^2 y_{j_2}
            \end{array}
            \hspace{-1mm}\right], \\[1.1mm]
        \EE[\eta_{j_1}^3 \eta_{j_2} \eta_{j_3} \eta_{j_4}]
        &= y_{j_1} y_{j_2} y_{j_3} y_{j_4} \big[(r + 3)^{(4)} + 3 (r + 4)^{(5)} y_{j_1} + (r + 5)^{(6)} y_{j_1}^2\big], \\[1.1mm]
        \EE[\eta_{j_1}^2 \eta_{j_2}^2 \eta_{j_3}^2]
        &= y_{j_1} y_{j_2} y_{j_3} \left[\hspace{-1mm}
            \begin{array}{l}
                (r + 2)^{(3)} + (r + 3)^{(4)} (y_{j_1} + y_{j_2} + y_{j_3}) \\[0.49mm]
                + (r + 4)^{(5)} (y_{j_1} y_{j_2} + y_{j_1} y_{j_3} + y_{j_2} y_{j_3}) + (r + 5)^{(6)} y_{j_1} y_{j_2} y_{j_3}
            \end{array}
            \hspace{-1mm}\right], \\[1.1mm]
        \EE[\eta_{j_1}^2 \eta_{j_2}^2 \eta_{j_3} \eta_{j_4}]
        &= y_{j_1} y_{j_2} y_{j_3} y_{j_4} \big[(r + 3)^{(4)} + (r + 4)^{(5)} (y_{j_1} + y_{j_2}) + (r + 5)^{(6)} y_{j_1} y_{j_2}\big], \\[1.1mm]
        \EE[\eta_{j_1}^2 \eta_{j_2} \eta_{j_3} \eta_{j_4} \eta_{j_5}]
        &= y_{j_1} y_{j_2} y_{j_3} y_{j_4} y_{j_5} \big[(r + 4)^{(5)} + (r + 5)^{(6)} y_{j_1}\big], \\[1.1mm]
        \EE[\eta_{j_1} \eta_{j_2} \eta_{j_3} \eta_{j_4} \eta_{j_5} \eta_{j_6}]
        &= y_{j_1} y_{j_2} y_{j_3} y_{j_4} y_{j_5} y_{j_6} (r + 5)^{(6)}.
    \end{align*}

    \vspace{2mm}
    \noindent
    \underline{7th order:} For different $j_1, j_2, j_3, j_4, j_5, j_6, j_7\in \{1,\dots,d\}$,
    \begin{align*}
        \EE[\eta_{j_1}^7]
        &= y_{j_1} \left[\hspace{-1mm}
            \begin{array}{l}
                r + 63 (r + 1)^{(2)} y_{j_1} + 301 (r + 2)^{(3)} y_{j_1}^2 \\[0.49mm]
                + 350 (r + 3)^{(4)} y_{j_1}^3 + 140 (r + 4)^{(5)} y_{j_1}^4 \\[0.49mm]
                + 21 (r + 5)^{(6)} y_{j_1}^5 + (r + 6)^{(7)} y_{j_1}^6
            \end{array}
            \hspace{-1mm}\right], \\[1.1mm]
        \EE[\eta_{j_1}^6 \eta_{j_2}]
        &= y_{j_1} y_{j_2} \left[\hspace{-1mm}
            \begin{array}{l}
                (r + 1)^{(2)} + 31 (r + 2)^{(3)} y_{j_1} + 90 (r + 3)^{(4)} y_{j_1}^2 \\[0.49mm]
                + 65 (r + 4)^{(5)} y_{j_1}^3 + 15 (r + 5)^{(6)} y_{j_1}^4 + (r + 6)^{(7)} y_{j_1}^5
            \end{array}
            \hspace{-1mm}\right], \\[1.1mm]
        \EE[\eta_{j_1}^5 \eta_{j_2}^2]
        &= y_{j_1} y_{j_2} \left[\hspace{-1mm}
            \begin{array}{l}
                (r + 1)^{(2)} + (r + 2)^{(3)} (15 y_{j_1} + y_{j_2}) \\[0.49mm]
                + (r + 3)^{(4)} (25 y_{j_1}^2 + 15 y_{j_1} y_{j_2}) \\[0.49mm]
                + (r + 4)^{(5)} (10 y_{j_1}^3 + 25 y_{j_1}^2 y_{j_2}) \\[0.49mm]
                + (r + 5)^{(6)} (y_{j_1}^4 + 10 y_{j_1}^3 y_{j_2}) + (r + 6)^{(7)} y_{j_1}^4 y_{j_2}
            \end{array}
            \hspace{-1mm}\right], \\[1.1mm]
        \EE[\eta_{j_1}^5 \eta_{j_2} \eta_{j_3}]
        &= y_{j_1} y_{j_2} y_{j_3} \left[\hspace{-1mm}
            \begin{array}{l}
                (r + 2)^{(3)} + 15 (r + 3)^{(4)} y_{j_1} + 25 (r + 4)^{(5)} y_{j_1}^2 \\[0.49mm]
                + 10 (r + 5)^{(6)} y_{j_1}^3 + (r + 6)^{(7)} y_{j_1}^4
            \end{array}
            \hspace{-1mm}\right], \\[1.1mm]
        \EE[\eta_{j_1}^4 \eta_{j_2}^3]
        &= y_{j_1} y_{j_2} \left[\hspace{-1mm}
            \begin{array}{l}
                (r + 1)^{(2)} + (r + 2)^{(3)} (7 y_{j_1} + 3 y_{j_2}) \\[0.49mm]
                + (r + 3)^{(4)} (6 y_{j_1}^2 + 21 y_{j_1} y_{j_2} + y_{j_2}^2) \\[0.49mm]
                + (r + 4)^{(5)} (y_{j_1}^3 + 18 y_{j_1}^2 y_{j_2} + 7 y_{j_1} y_{j_2}^2) \\[0.49mm]
                + (r + 5)^{(6)} (3 y_{j_1}^3 y_{j_2} + 6 y_{j_1}^2 y_{j_2}^2) + (r + 6)^{(7)} y_{j_1}^3 y_{j_2}^2
            \end{array}
            \hspace{-1mm}\right], \\[1.1mm]
        \EE[\eta_{j_1}^4 \eta_{j_2}^2 \eta_{j_3}]
        &= y_{j_1} y_{j_2} y_{j_3} \left[\hspace{-1mm}
            \begin{array}{l}
                (r + 2)^{(3)} + (r + 3)^{(4)} (7 y_{j_1} + y_{j_2}) \\[0.49mm]
                + (r + 4)^{(5)} (6 y_{j_1}^2 + 7 y_{j_1} y_{j_2}) \\[0.49mm]
                + (r + 5)^{(6)} (y_{j_1}^3 + 6 y_{j_1}^2 y_{j_2}) + (r + 6)^{(7)} y_{j_1}^3 y_{j_2}
            \end{array}
            \hspace{-1mm}\right], \\[1.1mm]
        \EE[\eta_{j_1}^4 \eta_{j_2} \eta_{j_3} \eta_{j_4}]
        &= y_{j_1} y_{j_2} y_{j_3} y_{j_4} \left[\hspace{-1mm}
            \begin{array}{l}
                (r + 3)^{(4)} + 7 (r + 4)^{(5)} y_{j_1} \\[0.49mm]
                + 6 (r + 5)^{(6)} y_{j_1}^2 + (r + 6)^{(7)} y_{j_1}^3
            \end{array}
            \hspace{-1mm}\right], \\[1.1mm]
        \EE[\eta_{j_1}^3 \eta_{j_2}^3 \eta_{j_3}]
        &= y_{j_1} y_{j_2} y_{j_3} \left[\hspace{-1mm}
            \begin{array}{l}
                (r + 2)^{(3)} + (r + 3)^{(4)} (3 y_{j_1} + 3 y_{j_2}) \\[0.49mm]
                + (r + 4)^{(5)} (y_{j_1}^2 + 9 y_{j_1} y_{j_2} + y_{j_2}^2) \\[0.49mm]
                + (r + 5)^{(6)} (3 y_{j_1}^2 y_{j_2} + 3 y_{j_1} y_{j_2}^2) + (r + 6)^{(7)} y_{j_1}^2 y_{j_2}^2
            \end{array}
            \hspace{-1mm}\right], \\[1.1mm]
        \EE[\eta_{j_1}^3 \eta_{j_2}^2 \eta_{j_3}^2]
        &= y_{j_1} y_{j_2} y_{j_3} \left[\hspace{-1mm}
            \begin{array}{l}
                (r + 2)^{(3)} + (r + 3)^{(4)} (3 y_{j_1} + y_{j_2} + y_{j_3}) \\[0.49mm]
                + (r + 4)^{(5)} (y_{j_1}^2 + 3 y_{j_1} y_{j_2} + 3 y_{j_1} y_{j_3} + y_{j_2} y_{j_3}) \\[0.49mm]
                + (r + 5)^{(6)} (y_{j_1}^2 y_{j_2} + y_{j_1}^2 y_{j_3} + 3 y_{j_1} y_{j_2} y_{j_3}) \\[0.49mm]
                + (r + 6)^{(7)} y_{j_1}^2 y_{j_2} y_{j_3}
            \end{array}
            \hspace{-1mm}\right], \\[1.1mm]
        \EE[\eta_{j_1}^3 \eta_{j_2}^2 \eta_{j_3} \eta_{j_4}]
        &= y_{j_1} y_{j_2} y_{j_3} y_{j_4} \left[\hspace{-1mm}
            \begin{array}{l}
                (r + 3)^{(4)} + (r + 4)^{(5)} (3 y_{j_1} + y_{j_2}) \\[0.49mm]
                + (r + 5)^{(6)} (y_{j_1}^2 + 3 y_{j_1} y_{j_2}) + (r + 6)^{(7)} y_{j_1}^2 y_{j_2}
            \end{array}
            \hspace{-1mm}\right], \\[1.1mm]
        \EE[\eta_{j_1}^3 \eta_{j_2} \eta_{j_3} \eta_{j_4} \eta_{j_5}]
        &= y_{j_1} y_{j_2} y_{j_3} y_{j_4} y_{j_5} \big[(r + 4)^{(5)} + 3 (r + 5)^{(6)} y_{j_1} + (r + 6)^{(7)} y_{j_1}^2\big], \\[1.1mm]
        \EE[\eta_{j_1}^2 \eta_{j_2}^2 \eta_{j_3}^2 \eta_{j_4}]
        &= y_{j_1} y_{j_2} y_{j_3} y_{j_4} \left[\hspace{-1mm}
            \begin{array}{l}
                (r + 3)^{(4)} + (r + 4)^{(5)} (y_{j_1} + y_{j_2} + y_{j_3}) \\[0.49mm]
                + (r + 5)^{(6)} (y_{j_1} y_{j_2} + y_{j_1} y_{j_3} + y_{j_2} y_{j_3}) \\[0.49mm]
                + (r + 6)^{(7)} y_{j_1} y_{j_2} y_{j_3}
            \end{array}
            \hspace{-1mm}\right], \\[1.1mm]
        \EE[\eta_{j_1}^2 \eta_{j_2}^2 \eta_{j_3} \eta_{j_4} \eta_{j_5}]
        &= y_{j_1} y_{j_2} y_{j_3} y_{j_4} y_{j_5} \left[\hspace{-1mm}
            \begin{array}{l}
                (r + 4)^{(5)} + (r + 5)^{(6)} (y_{j_1} + y_{j_2}) \\[0.49mm]
                + (r + 6)^{(7)} y_{j_1} y_{j_2}
            \end{array}
            \hspace{-1mm}\right], \\[1.1mm]
        \EE[\eta_{j_1}^2 \eta_{j_2} \eta_{j_3} \eta_{j_4} \eta_{j_5} \eta_{j_6}]
        &= y_{j_1} y_{j_2} y_{j_3} y_{j_4} y_{j_5} y_{j_6} \big[(r + 5)^{(6)} + (r + 6)^{(7)} y_{j_1}\big], \\[1.1mm]
        \EE[\eta_{j_1} \eta_{j_2} \eta_{j_3} \eta_{j_4} \eta_{j_5} \eta_{j_6} \eta_{j_7}]
        &= y_{j_1} y_{j_2} y_{j_3} y_{j_4} y_{j_5} y_{j_6} y_{j_7} (r + 6)^{(7)}.
    \end{align*}

    \vspace{2mm}
    \noindent
    \underline{8th order:} For different $j_1, j_2, j_3, j_4, j_5, j_6, j_7, j_8\in \{1,\dots,d\}$,
    \begin{align*}
        \EE[\eta_{j_1}^8]
        &= y_{j_1} \left[\hspace{-1mm}
            \begin{array}{l}
                r + 127 (r + 1)^{(2)} y_{j_1} + 966 (r + 2)^{(3)} y_{j_2}^2 \\[0.49mm]
                + 1701 (r + 3)^{(4)} y_{j_1}^3 + 1050 (r + 4)^{(5)} y_{j_1}^4 \\[0.49mm]
                + 266 (r + 5)^{(6)} y_{j_1}^5 + 28 (r + 6)^{(7)} y_{j_1}^6 \\[0.49mm]
                + (r + 7)^{(8)} y_{j_1}^7
            \end{array}
            \hspace{-1mm}\right], \\[1.1mm]
        \EE[\eta_{j_1}^7 \eta_{j_2}]
        &= y_{j_1} y_{j_2} \left[\hspace{-1mm}
            \begin{array}{l}
                r + 63 (r + 2)^{(3)} y_{j_1} + 301 (r + 3)^{(4)} y_{j_1}^2 \\[0.49mm]
                + 350 (r + 4)^{(5)} y_{j_1}^3 + 140 (r + 5)^{(6)} y_{j_1}^4 \\[0.49mm]
                + 21 (r + 6)^{(7)} y_{j_1}^5 + (r + 7)^{(8)} y_{j_1}^6
            \end{array}
            \hspace{-1mm}\right], \\[1.1mm]
        \EE[\eta_{j_1}^6 \eta_{j_2}^2]
        &= y_{j_1} y_{j_2} \left[\hspace{-1mm}
            \begin{array}{l}
                (r + 1)^{(2)} + (r + 2)^{(3)} (31 y_{j_1} + y_{j_2}) \\[0.49mm]
                + (r + 3)^{(4)} (90 y_{j_1}^2 + 31 y_{j_1} y_{j_2}) \\[0.49mm]
                + (r + 4)^{(5)} (65 y_{j_1}^3 + 90 y_{j_1}^2 y_{j_2}) \\[0.49mm]
                + (r + 5)^{(6)} (15 y_{j_1}^4 + 65 y_{j_1}^3 y_{j_2}) \\[0.49mm]
                + (r + 6)^{(7)} (y_{j_1}^5 + 15 y_{j_1}^4 y_{j_2}) + (r + 7)^{(8)} y_{j_1}^5 y_{j_2}
            \end{array}
            \hspace{-1mm}\right], \\[1.1mm]
        \EE[\eta_{j_1}^6 \eta_{j_2} \eta_{j_3}]
        &= y_{j_1} y_{j_2} y_{j_3} \left[\hspace{-1mm}
            \begin{array}{l}
                (r + 2)^{(3)} + 31 (r + 3)^{(4)} y_{j_1} + 90 (r + 4)^{(5)} y_{j_1}^2 \\[0.49mm]
                + 65 (r + 5)^{(6)} y_{j_1}^3 + 15 (r + 6)^{(7)} y_{j_1}^4 + (r + 7)^{(8)} y_{j_1}^5
            \end{array}
            \hspace{-1mm}\right], \\[1.1mm]
        \EE[\eta_{j_1}^5 \eta_{j_2}^3]
        &= y_{j_1} y_{j_2} \left[\hspace{-1mm}
            \begin{array}{l}
                (r + 1)^{(2)} + (r + 2)^{(3)} (15 y_{j_1} + 3 y_{j_2}) \\[0.49mm]
                + (r + 3)^{(4)} (25 y_{j_1}^2 + 45 y_{j_1} y_{j_2} + y_{j_2}^2) \\[0.49mm]
                + (r + 4)^{(5)} (10 y_{j_1}^3 + 75 y_{j_1}^2 y_{j_2} + 15 y_{j_1} y_{j_2}^2) \\[0.49mm]
                + (r + 5)^{(6)} (y_{j_4}^4 + 30 y_{j_1}^3 y_{j_2} + 25 y_{j_1}^2 y_{j_2}^2) \\[0.49mm]
                + (r + 6)^{(7)} (3 y_{j_1}^4 y_{j_2} + 10 y_{j_1}^3 y_{j_2}^2) + (r + 7)^{(8)} y_{j_1}^4 y_{j_2}^2
            \end{array}
            \hspace{-1mm}\right], \\[1.1mm]
        \EE[\eta_{j_1}^5 \eta_{j_2}^2 \eta_{j_3}]
        &= y_{j_1} y_{j_2} y_{j_3} \left[\hspace{-1mm}
            \begin{array}{l}
                (r + 2)^{(3)} + (r + 3)^{(4)} (15 y_{j_1} + y_{j_2}) \\[0.49mm]
                + (r + 4)^{(5)} (25 y_{j_1}^2 + 15 y_{j_1} y_{j_2}) \\[0.49mm]
                + (r + 5)^{(6)} (10 y_{j_1}^3 + 25 y_{j_1}^2 y_{j_2}) \\[0.49mm]
                + (r + 6)^{(7)} (y_{j_1}^4 + 10 y_{j_1}^3 y_{j_2}) + (r + 7)^{(8)} y_{j_1}^4 y_{j_2}
            \end{array}
            \hspace{-1mm}\right], \\[1.1mm]
        \EE[\eta_{j_1}^5 \eta_{j_2} \eta_{j_3} \eta_{j_4}]
        &= y_{j_1} y_{j_2} y_{j_3} y_{j_4} \left[\hspace{-1mm}
            \begin{array}{l}
                (r + 3)^{(4)} + 15 (r + 4)^{(5)} y_{j_1} + 25 (r + 5)^{(6)} y_{j_1}^2 \\[0.49mm]
                + 10 (r + 6)^{(7)} y_{j_1}^3 + (r + 7)^{(8)} y_{j_1}^4
            \end{array}
            \hspace{-1mm}\right], \\[1.1mm]
        \EE[\eta_{j_1}^4 \eta_{j_2}^4]
        &= y_{j_1} y_{j_2} \left[\hspace{-1mm}
            \begin{array}{l}
                (r + 1)^{(2)} + (r + 2)^{(3)} (7 y_{j_1} + 7 y_{j_2}) \\[0.49mm]
                + (r + 3)^{(4)} (6 y_{j_1}^2 + 49 y_{j_1} y_{j_2} + 6 y_{j_2}^2) \\[0.49mm]
                + (r + 4)^{(5)} (y_{j_1}^3 + 42 y_{j_1}^2 y_{j_2} + 42 y_{j_1} y_{j_2}^2 + y_{j_2}^3) \\[0.49mm]
                + (r + 5)^{(6)} (7 y_{j_1}^3 y_{j_2} + 36 y_{j_1}^2 y_{j_2}^2 + 7 y_{j_1} y_{j_2}^3) \\[0.49mm]
                + (r + 6)^{(7)} (6 y_{j_1}^3 y_{j_2}^2 + 6 y_{j_1}^2 y_{j_2}^3) + (r + 7)^{(8)} y_{j_1}^3 y_{j_2}^3
            \end{array}
            \hspace{-1mm}\right], \\[1.1mm]
        \EE[\eta_{j_1}^4 \eta_{j_2}^3 \eta_{j_3}]
        &= y_{j_1} y_{j_2} y_{j_3} \left[\hspace{-1mm}
            \begin{array}{l}
                (r + 2)^{(3)} + (r + 3)^{(4)} (7 y_{j_1} + 3 y_{j_2}) \\[0.49mm]
                + (r + 4)^{(5)} (6 y_{j_1}^2 + 21 y_{j_1} y_{j_2} + y_{j_2}^2) \\[0.49mm]
                + (r + 5)^{(6)} (y_{j_1}^3 + 18 y_{j_1}^2 y_{j_2} + 7 y_{j_1} y_{j_2}^2) \\[0.49mm]
                + (r + 6)^{(7)} (3 y_{j_1}^3 y_{j_2} + 6 y_{j_1}^2 y_{j_2}^2) + (r + 7)^{(8)} y_{j_1}^3 y_{j_2}^2
            \end{array}
            \hspace{-1mm}\right], \\[1.1mm]
        \EE[\eta_{j_1}^4 \eta_{j_2}^2 \eta_{j_3}^2]
        &= y_{j_1} y_{j_2} y_{j_3} \left[\hspace{-1mm}
            \begin{array}{l}
                (r + 2)^{(3)} + (r + 3)^{(4)} (7 y_{j_1} + y_{j_2} + y_{j_3}) \\[0.49mm]
                + (r + 4)^{(5)} (6 y_{j_1}^2 + 7 y_{j_1} y_{j_2} + 7 y_{j_1} y_{j_3} + y_{j_2} y_{j_3}) \\[0.49mm]
                + (r + 5)^{(6)} (y_{j_1}^3 + 6 y_{j_1}^2 y_{j_2} + 6 y_{j_1}^2 y_{j_3} + 7 y_{j_1} y_{j_2} y_{j_3}) \\[0.49mm]
                + (r + 6)^{(7)} (y_{j_1}^3 y_{j_2} + y_{j_1}^3 y_{j_3} + 6 y_{j_1}^2 y_{j_2} y_{j_3}) \\[0.49mm]
                + (r + 7)^{(8)} y_{j_1}^3 y_{j_2} y_{j_3}
            \end{array}
            \hspace{-1mm}\right], \\[1.1mm]
        \EE[\eta_{j_1}^4 \eta_{j_2}^2 \eta_{j_3} \eta_{j_4}]
        &= y_{j_1} y_{j_2} y_{j_3} y_{j_4} \left[\hspace{-1mm}
            \begin{array}{l}
                (r + 3)^{(4)} + (r + 4)^{(5)} (7 y_{j_1} + y_{j_2}) \\[0.49mm]
                + (r + 5)^{(6)} (6 y_{j_1}^2 + 7 y_{j_1} y_{j_2}) \\[0.49mm]
                + (r + 6)^{(7)} (y_{j_1}^3 + 6 y_{j_1}^2 y_{j_2}) + (r + 7)^{(8)} y_{j_1}^3 y_{j_2}
            \end{array}
            \hspace{-1mm}\right], \\[1.1mm]
        \EE[\eta_{j_1}^4 \eta_{j_2} \eta_{j_3} \eta_{j_4} \eta_{j_5}]
        &= y_{j_1} y_{j_2} y_{j_3} y_{j_4} y_{j_5} \left[\hspace{-1mm}
            \begin{array}{l}
                (r + 4)^{(5)} + 7 (r + 5)^{(6)} y_{j_1} + 6 (r + 6)^{(7)} y_{j_1}^2 \\[0.49mm]
                + (r + 7)^{(8)} y_{j_1}^3
            \end{array}
            \hspace{-1mm}\right], \\[1.1mm]
        \EE[\eta_{j_1}^3 \eta_{j_2}^3 \eta_{j_3}^2]
        &= y_{j_1} y_{j_2} y_{j_3} \left[\hspace{-1mm}
            \begin{array}{l}
                (r + 2)^{(3)} + (r + 3)^{(4)} (3 y_{j_1} + 3 y_{j_2} + y_{j_3}) \\[0.49mm]
                + (r + 4)^{(5)} \left(\hspace{-1mm}
                    \begin{array}{l}
                        y_{j_1}^2 + y_{j_2}^2 + 3 y_{j_1} y_{j_3} \\[1mm]
                        + 3 y_{j_2} y_{j_3} + 9 y_{j_1} y_{j_2}
                    \end{array}
                    \hspace{-1mm}\right) \\[0.49mm]
                + (r + 5)^{(6)} \left(\hspace{-1mm}
                    \begin{array}{l}
                        y_{j_1}^2 y_{j_3} + y_{j_2}^2 y_{j_3} + 3 y_{j_1}^2 y_{j_2} \\[1mm]
                        + 3 y_{j_1} y_{j_2}^2 + 9 y_{j_1} y_{j_2} y_{j_3}
                    \end{array}
                    \hspace{-1mm}\right) \\[0.49mm]
                + (r + 6)^{(7)} (y_{j_1}^2 y_{j_2}^2 + 3 y_{j_1}^2 y_{j_2} y_{j_3} + 3 y_{j_1} y_{j_2}^2 y_{j_3}) \\[0.49mm]
                + (r + 7)^{(8)} y_{j_1}^2 y_{j_2}^2 y_{j_3}
            \end{array}
            \hspace{-1mm}\right], \\[1.1mm]
        \EE[\eta_{j_1}^3 \eta_{j_2}^3 \eta_{j_3} \eta_{j_4}]
        &= y_{j_1} y_{j_2} y_{j_3} y_{j_4} \left[\hspace{-1mm}
            \begin{array}{l}
                (r + 3)^{(4)} + (r + 4)^{(5)} (3 y_{j_1} + 3 y_{j_2}) \\[0.49mm]
                + (r + 5)^{(6)} (y_{j_1}^2 + 9 y_{j_1} y_{j_2} + y_{j_2}^2) \\[0.49mm]
                + (r + 6)^{(7)} (3 y_{j_1}^2 y_{j_2} + 3 y_{j_1} y_{j_2}^2) \\[0.49mm]
                + (r + 7)^{(8)} y_{j_1}^2 y_{j_2}^2
            \end{array}
            \hspace{-1mm}\right], \\[1.1mm]
        \EE[\eta_{j_1}^3 \eta_{j_2}^2 \eta_{j_3}^2 \eta_{j_4}]
        &= y_{j_1} y_{j_2} y_{j_3} y_{j_4} \left[\hspace{-1mm}
            \begin{array}{l}
                (r + 3)^{(4)} + (r + 4)^{(5)} (3 y_{j_1} + y_{j_2} + y_{j_3}) \\[0.49mm]
                + (r + 5)^{(6)} (3 y_{j_1} y_{j_2} + 3 y_{j_1} y_{j_3} + y_{j_2} y_{j_3}) \\[0.49mm]
                + (r + 6)^{(7)} (y_{j_1}^2 y_{j_2} + y_{j_1}^2 y_{j_3} + 3 y_{j_1} y_{j_2} y_{j_3}) \\[0.49mm]
                + (r + 7)^{(8)} y_{j_1}^2 y_{j_2} y_{j_3}
            \end{array}
            \hspace{-1mm}\right], \\[1.1mm]
        \EE[\eta_{j_1}^3 \eta_{j_2}^2 \eta_{j_3} \eta_{j_4} \eta_{j_5}]
        &= y_{j_1} y_{j_2} y_{j_3} y_{j_4} y_{j_5} \left[\hspace{-1mm}
            \begin{array}{l}
                (r + 4)^{(5)} + (r + 5)^{(6)} (3 y_{j_1} + y_{j_2}) \\[0.49mm]
                + (r + 6)^{(7)} (y_{j_1}^2 + 3 y_{j_1} y_{j_2}) \\[0.49mm]
                + (r + 7)^{(8)} y_{j_1}^2 y_{j_2}
            \end{array}
            \hspace{-1mm}\right], \\[1.1mm]
        \EE[\eta_{j_1}^3 \eta_{j_2} \eta_{j_3} \eta_{j_4} \eta_{j_5} \eta_{j_6}]
        &= y_{j_1} y_{j_2} y_{j_3} y_{j_4} y_{j_5} y_{j_6} \big[(r + 5)^{(6)} + 3 (r + 6)^{(7)} y_{j_1} + (r + 7)^{(8)} y_{j_1}^2\big], \\[1.1mm]
        \EE[\eta_{j_1}^2 \eta_{j_2}^2 \eta_{j_3}^2 \eta_{j_4}^2]
        &= y_{j_1} y_{j_2} y_{j_3} y_{j_4} \left[\hspace{-1mm}
            \begin{array}{l}
                (r + 3)^{(4)} + (r + 4)^{(5)} (y_{j_1} + y_{j_2} + y_{j_3} + y_{j_4}) \\[0.49mm]
                + (r + 5)^{(6)} \left(\hspace{-1mm}
                    \begin{array}{l}
                        y_{j_1} y_{j_2} + y_{j_1} y_{j_3} + y_{j_1} y_{j_4} \\[1mm]
                        + y_{j_2} y_{j_3} + y_{j_2} y_{j_4} + y_{j_3} y_{j_4}
                    \end{array}
                    \hspace{-1mm}\right) \\[0.49mm]
                + (r + 6)^{(7)} \left(\hspace{-1mm}
                    \begin{array}{l}
                        y_{j_1} y_{j_2} y_{j_3} + y_{j_1} y_{j_2} y_{j_4} \\[1mm]
                        + y_{j_1} y_{j_3} y_{j_4} + y_{j_2} y_{j_3} y_{j_4}
                    \end{array}
                    \hspace{-1mm}\right) \\[0.49mm]
                + (r + 7)^{(8)} y_{j_1} y_{j_2} y_{j_3} y_{j_4}
            \end{array}
            \hspace{-1mm}\right], \\[1.1mm]
        \EE[\eta_{j_1}^2 \eta_{j_2}^2 \eta_{j_3}^2 \eta_{j_4} \eta_{j_5}]
        &= y_{j_1} y_{j_2} y_{j_3} y_{j_4} y_{j_5} \left[\hspace{-1mm}
            \begin{array}{l}
                (r + 4)^{(5)} + (r + 5)^{(6)} (y_{j_1} + y_{j_2} + y_{j_3}) \\[0.49mm]
                + (r + 6)^{(7)} (y_{j_1} y_{j_2} + y_{j_1} y_{j_3} + y_{j_2} y_{j_3}) \\[0.49mm]
                + (r + 7)^{(8)} y_{j_1} y_{j_2} y_{j_3}
            \end{array}
            \hspace{-1mm}\right], \\[1.1mm]
        \EE[\eta_{j_1}^2 \eta_{j_2}^2 \eta_{j_3} \eta_{j_4} \eta_{j_5} \eta_{j_6}]
        &= y_{j_1} y_{j_2} y_{j_3} y_{j_4} y_{j_5} y_{j_6} \left[\hspace{-1mm}
            \begin{array}{l}
                (r + 5)^{(6)} + (r + 6)^{(7)} (y_{j_1} + y_{j_2}) \\[0.49mm]
                + (r + 7)^{(8)} y_{j_1} y_{j_2}
            \end{array}
            \hspace{-1mm}\right], \\[1.1mm]
        \EE[\eta_{j_1}^2 \eta_{j_2} \eta_{j_3} \eta_{j_4} \eta_{j_5} \eta_{j_6} \eta_{j_7}]
        &= y_{j_1} y_{j_2} y_{j_3} y_{j_4} y_{j_5} y_{j_6} y_{j_7} \big[(r + 6)^{(7)} + (r + 7)^{(8)} y_{j_1}\big], \\[1.1mm]
        \EE[\eta_{j_1} \eta_{j_2} \eta_{j_3} \eta_{j_4} \eta_{j_5} \eta_{j_6} \eta_{j_7} \eta_{j_8}]
        &= y_{j_1} y_{j_2} y_{j_3} y_{j_4} y_{j_5} y_{j_6} y_{j_7} y_{j_8} (r + 7)^{(8)}.
    \end{align*}

    \subsection{Computation of the Central Moments up to the Fourth~Order}

    By combining the results of Section~\ref{sec:noncentral} with some algebraic manipulations, we are now able to calculate the central moments explicitly. The simplifications we apply to arrive at the final boxed expressions below were performed using \texttt{Mathematica}. We use a symbolic calculator like \texttt{Mathematica} to do the simplifications because many terms cancel each other out in every expression; it would be virtually impossible to do the simplifications by hand without making mistakes. While our methodology allows us to obtain simplified formulas for the central moments up to any order in principle (assuming we calculate explicit expressions for the appropriate higher order non-central moments in Section~\ref{sec:noncentral}), it would be quite time-consuming for us to input the base formula for the central moments as a function of the non-central moments in \texttt{Mathematica} and let \texttt{Mathematica} do the simplifications beyond the fourth~order. Therefore, for the sake of conciseness, we only present explicit simplified formulas for the central moments up to the fourth~order below. It is worth noting that the numerical formulas we developed in Section~\ref{sec:numerical.codes} are fast for higher orders (i.e., beyond the fourth~order) if the categorical probabilities $x_i$ are known; otherwise, \texttt{Mathematica} has trouble calculating for unknown values of $x_i$'s (i.e., \texttt{Mathematica} %MDPI: Is this text format necessary? (yes)
    has trouble getting simplified general expressions by itself. This is why our approach below~is~necessary.

    \vspace{3mm}
    \noindent
    %MDPI: Please check if boxes can be removed from equations. (DO NOT remove the boxes)
    \underline{2nd order:} For different $j_1, j_2\in \{1,\dots,d\}$,
    \begin{align*}
        \EE[(\eta_{j_1} - \EE[\eta_{j_1}])^2]
        &= \EE[\eta_{j_1}^2] - (\EE[\eta_{j_1}])^2 \\
        &= y_{j_1} \big[r + (r + 1)^{(2)} y_{j_1}\big] - r^2 y_{j_1}^2 \\
        &= \boxed{r y_{j_1} (1 + y_{j_1})} \\[3.1mm]
        \EE[(\eta_{j_1} - \EE[\eta_{j_1}]) (\eta_{j_2} - \EE[\eta_{j_2}])]
        &= \EE[\eta_{j_1} \eta_{j_2}] - \EE[\eta_{j_1}] \EE[\eta_{j_2}] \\
        &= (r + 1)^{(2)} y_{j_1} y_{j_2} - r y_{j_1} r y_{j_2} \\
        &= \boxed{r y_{j_1} y_{j_2}}.
    \end{align*}
\clearpage
    \vspace{2mm}
    \noindent
    \underline{3rd order:} For different $j_1, j_2, j_3\in \{1,\dots,d\}$,
    \begin{align*}
        &\EE[(\eta_{j_1} - \EE[\eta_{j_1}])^3] \\[1.1mm]
        &\qquad= \EE[\eta_{j_1}^3] - 3 \, \EE[\eta_{j_1}^2] \EE[\eta_{j_1}] + 2 \, (\EE[\eta_{j_1}])^3 \notag \\[1.1mm]
        &\qquad= y_{j_1} \big[r + 3 (r + 1)^{(2)} y_{j_1} + (r + 2)^{(3)} y_{j_1}^2\big] - 3 y_{j_1} \big[r + (r + 1)^{(2)} y_{j_1}\big] r y_{j_1} + 2 r^3 y_{j_1}^3 \\
        &\qquad= \boxed{r y_{j_1} (1 + 3 y_{j_1} + 2 y_{j_1}^2)} \\[3.1mm]
        %%%
        &\EE[(\eta_{j_1} - \EE[\eta_{j_1}])^2 (\eta_{j_2} - \EE[\eta_{j_2}])] \notag \\[1.1mm]
        &\qquad= \EE[\eta_{j_1}^2 \eta_{j_2}] - \EE[\eta_{j_1}^2] \EE[\eta_{j_2}] - 2 \, \EE[\eta_{j_1} \eta_{j_2}] \EE[\eta_{j_1}] + 2 \, (\EE[\eta_{j_1}])^2 \EE[\eta_{j_2}] \notag \\[1.1mm]
        &\qquad= y_{j_1} y_{j_2} \big[(r + 1)^{(2)} + (r + 2)^{(3)} y_{j_1}\big] - y_{j_1} \big[r + (r + 1)^{(2)} y_{j_1}\big] r y_{j_2} \notag \\
        &\qquad\quad- 2 (r + 1)^{(2)} y_{j_1} y_{j_2} r y_{j_1} + 2 r^2 y_{j_1}^2 r y_{j_2} \\
        &\qquad= \boxed{r y_{j_1} (1 + 2 y_{j_1}) y_{j_2}} \\[3.1mm]
        %%%
        &\EE[(\eta_{j_1} - \EE[\eta_{j_1}]) (\eta_{j_2} - \EE[\eta_{j_2}]) (\eta_{j_3} - \EE[\eta_{j_3}])] \notag \\[1.1mm]
        &\qquad= \EE[\eta_{j_1} \eta_{j_2} \eta_{j_3}] - \EE[\eta_{j_1} \eta_{j_2}] \EE[\eta_{j_3}] - \EE[\eta_{j_1} \eta_{j_3}] \EE[\eta_{j_2}] \notag \\[1.1mm]
        &\qquad\quad- \EE[\eta_{j_2} \eta_{j_3}] \EE[\eta_{j_1}] + 2 \, \EE[\eta_{j_1}] \EE[\eta_{j_2}] \EE[\eta_{j_3}] \notag \\[1.1mm]
        &\qquad= (r + 2)^{(3)} y_{j_1} y_{j_2} y_{j_3} - (r + 1)^{(2)} y_{j_1} y_{j_2} r y_{j_3} - (r + 1)^{(2)} y_{j_1} y_{j_3} r y_{j_2} \notag \\
        &\qquad\quad- (r + 1)^{(2)} y_{j_2} y_{j_3} r y_{j_1} + 2 r^3 y_{j_1} y_{j_2} y_{j_3} \\
        &\qquad= \boxed{2 r y_{j_1} y_{j_2} y_{j_3}}.
    \end{align*}

    \vspace{2mm}
    \noindent
    \underline{4th order:} For different $j_1, j_2, j_3, j_4\in \{1,\dots,d\}$,
    \begin{align*}
        &\EE[(\eta_{j_1} - \EE[\eta_{j_1}])^4] \notag \\[1.1mm]
        &\qquad= \EE[\eta_{j_1}^4] - 4 \, \EE[\eta_{j_1}^3] \EE[\eta_{j_1}] + 6 \, \EE[\eta_{j_1}^2] (\EE[\eta_{j_1}])^2 - 3 \, (\EE[\eta_{j_1}])^4 \notag \\[1.1mm]
        &\qquad= y_{j_1} \big[r + 7 (r + 1)^{(2)} y_{j_1} + 6 (r + 2)^{(3)} y_{j_1}^2 + (r + 3)^{(4)} y_{j_1}^3\big] \notag \\
        &\qquad\quad- 4 y_{j_1} \big[r + 3 (r + 1)^{(2)} y_{j_1} + (r + 2)^{(3)} y_{j_1}^2\big] r y_{j_1} \notag \\
        &\qquad\quad+ 6 y_{j_1} \big[r + (r + 1)^{(2)} y_{j_1}\big] (r y_{j_1})^2 - 3 r^4 y_{j_1}^4 \\
        &\qquad= \boxed{r y_{j_1} (1 + y_{j_1}) (1 + 3 (2 + r) y_{j_1} + 3 (2 + r) y_{j_1}^2)} \\[3.1mm]
        %%%
        &\EE[(\eta_{j_1} - \EE[\eta_{j_1}])^3 (\eta_{j_2} - \EE[\eta_{j_2}])] \notag \\[1.1mm]
        &\qquad= \EE[\eta_{j_1}^3 \eta_{j_2}] - \EE[\eta_{j_1}^3] \EE[\eta_{j_2}] - 3 \, \EE[\eta_{j_1}^2 \eta_{j_2}] \EE[\eta_{j_1}] + 3 \, \EE[\eta_{j_1}^2] \EE[\eta_{j_1}] \EE[\eta_{j_2}] \notag \\[1.1mm]
        &\qquad\quad+ 3 \, \EE[\eta_{j_1} \eta_{j_2}] (\EE[\eta_{j_1}])^2 - 3 \, (\EE[\eta_{j_1}])^3 \EE[\eta_{j_2}] \notag \\[1.1mm]
        &\qquad= y_{j_1} y_{j_2} \big[(r + 1)^{(2)} + 3 (r + 2)^{(3)} y_{j_1} + (r + 3)^{(4)} y_{j_1}^2\big] \notag \\
        &\qquad\quad- y_{j_1} \big[r + 3 (r + 1)^{(2)} y_{j_1} + (r + 2)^{(3)} y_{j_1}^2\big] r y_{j_2} \notag \\
        &\qquad\quad- 3 y_{j_1} y_{j_2} \big[(r + 1)^{(2)} + (r + 2)^{(3)} y_{j_1}\big] r y_{j_1} + 3 y_{j_1} \big[r + (r + 1)^{(2)} y_{j_1}\big] r y_{j_1} r y_{j_2} \notag \\
        &\qquad\quad+ 3 (r + 1)^{(2)} y_{j_1} y_{j_2} r^2 y_{j_1}^2 - 3 r^3 y_{j_1}^3 r y_{j_2} \\
        &\qquad= \boxed{r y_{j_1} (1 + 3 (2 + r) y_{j_1} + 3 (2 + r) y_{j_1}^2) y_{j_2}} \\[3.1mm]
        %%%
        &\EE[(\eta_{j_1} - \EE[\eta_{j_1}])^2 (\eta_{j_2} - \EE[\eta_{j_2}])^2] \notag \\[1.1mm]
        &\qquad= \EE[\eta_{j_1}^2 \eta_{j_2}^2] - 2 \, \EE[\eta_{j_1}^2 \eta_{j_2}] \EE[\eta_{j_2}] - 2 \, \EE[\eta_{j_1} \eta_{j_2}^2] \EE[\eta_{j_1}] + \EE[\eta_{j_1}^2] (\EE[\eta_{j_2}])^2 + \EE[\eta_{j_2}^2] (\EE[\eta_{j_1}])^2 \notag \\[1.1mm]
        &\qquad\quad+ 4 \, \EE[\eta_{j_1} \eta_{j_2}] \EE[\eta_{j_1}] \EE[\eta_{j_2}] - 3 \, (\EE[\eta_{j_1}])^2 (\EE[\eta_{j_2}])^2 \notag \\[1.1mm]
        &\qquad= y_{j_1} y_{j_2} \big[(r + 1)^{(2)} + (r + 2)^{(3)} (y_{j_1} + y_{j_2}) + (r + 3)^{(4)} y_{j_1} y_{j_2}\big] \notag \\
        &\qquad\quad- 2 y_{j_1} y_{j_2} \big[(r + 1)^{(2)} + (r + 2)^{(3)} y_{j_1}\big] r y_{j_2} - 2 y_{j_1} y_{j_2} \big[(r + 1)^{(2)} + (r + 2)^{(3)} y_{j_2}\big] r y_{j_1} \notag \\
        &\qquad\quad+ y_{j_1} \big[r + (r + 1)^{(2)} y_{j_1}\big] r^2 y_{j_2}^2 + y_{j_2} \big[r + (r + 1)^{(2)} y_{j_2}\big] r^2 y_{j_1}^2 \notag \\
        &\qquad\quad+ 4 (r + 1)^{(2)} y_{j_1} y_{j_2} r y_{j_1} r y_{j_2} - 3 \, r^2 y_{j_1}^2 r^2 y_{j_2}^2 \\
        &\qquad= \boxed{r y_{j_1} y_{j_2} (1 + 2 y_{j_2} + y_{j_1} (2 + 6 y_{j_2}) + r (1 + y_{j_1} + y_{j_2} + 3 y_{j_1} y_{j_2}))} \\[3.1mm]
        %%%
        &\EE[(\eta_{j_1} - \EE[\eta_{j_1}])^2 (\eta_{j_2} - \EE[\eta_{j_2}]) (\eta_{j_3} - \EE[\eta_{j_3}])] \notag \\[1.1mm]
        &\qquad= \EE[\eta_{j_1}^2 \eta_{j_2} \eta_{j_3}] - \EE[\eta_{j_1}^2 \eta_{j_2}] \EE[\eta_{j_3}] - \EE[\eta_{j_1}^2 \eta_{j_3}] \EE[\eta_{j_2}] - 2 \, \EE[\eta_{j_1} \eta_{j_2} \eta_{j_3}] \EE[\eta_{j_1}] \notag \\[1.1mm]
        &\qquad\quad+ \EE[\eta_{j_1}^2] \EE[\eta_{j_2}] \EE[\eta_{j_3}] + 2 \, \EE[\eta_{j_1} \eta_{j_2}] \EE[\eta_{j_1}] \EE[\eta_{j_3}] + 2 \, \EE[\eta_{j_1} \eta_{j_3}] \EE[\eta_{j_1}] \EE[\eta_{j_2}] \notag \\[1.1mm]
        &\qquad\quad+ \EE[\eta_{j_2} \eta_{j_3}] (\EE[\eta_{j_1}])^2 - 3 \, (\EE[\eta_{j_1}])^2 \EE[\eta_{j_2}] \EE[\eta_{j_3}] \notag \\[1.1mm]
        &\qquad= y_{j_1} y_{j_2} y_{j_3} \big[(r + 2)^{(3)} + (r + 3)^{(4)} y_{j_1}\big] - y_{j_1} y_{j_2} \big[(r + 1)^{(2)} + (r + 2)^{(3)} y_{j_1}] r y_{j_3} \notag \\
        &\qquad\quad- y_{j_1} y_{j_3} \big[(r + 1)^{(2)} + (r + 2)^{(3)} y_{j_1}\big] r y_{j_2} - 2 (r + 2)^{(3)} y_{j_1} y_{j_2} y_{j_3} r y_{j_1} \notag \\
        &\qquad\quad+ y_{j_1} \big[r + (r + 1)^{(2)} y_{j_1}\big] r y_{j_2} r y_{j_3} + 2 (r + 1)^{(2)} y_{j_1} y_{j_2} r y_{j_1} r y_{j_3} \notag \\
        &\qquad\quad+ 2 (r + 1)^{(2)} y_{j_1} y_{j_3} r y_{j_1} r y_{j_2} + (r + 1)^{(2)} y_{j_2} y_{j_3} r^2 y_{j_1}^2 - 3 r^2 y_{j_1}^2 r y_{j_2} r y_{j_3} \\
        &\qquad= \boxed{r (2 + r) y_{j_1} (1 + 3 y_{j_1}) y_{j_2} y_{j_3}} \\[3.1mm]
        %%%
        &\EE[(\eta_{j_1} - \EE[\eta_{j_1}]) (\eta_{j_2} - \EE[\eta_{j_2}]) (\eta_{j_3} - \EE[\eta_{j_3}]) (\eta_{j_4} - \EE[\eta_{j_4}])] \notag \\[1.1mm]
        &\qquad= \EE[\eta_{j_1} \eta_{j_2} \eta_{j_3} \eta_{j_4}] - \EE[\eta_{j_1} \eta_{j_2} \eta_{j_3}] \EE[\eta_{j_4}] - \EE[\eta_{j_1} \eta_{j_2} \eta_{j_4}] \EE[\eta_{j_3}] - \EE[\eta_{j_1} \eta_{j_3} \eta_{j_4}] \EE[\eta_{j_2}] \notag \\[1.1mm]
        &\qquad\quad- \EE[\eta_{j_2} \eta_{j_3} \eta_{j_4}] \EE[\eta_{j_1}] + \EE[\eta_{j_1} \eta_{j_2}] \EE[\eta_{j_3}] \EE[\eta_{j_4}] + \EE[\eta_{j_1} \eta_{j_3}] \EE[\eta_{j_2}] \EE[\eta_{j_4}] \notag \\[1.1mm]
        &\qquad\quad+ \EE[\eta_{j_1} \eta_{j_4}] \EE[\eta_{j_2}] \EE[\eta_{j_3}] + \EE[\eta_{j_2} \eta_{j_3}] \EE[\eta_{j_1}] \EE[\eta_{j_4}] + \EE[\eta_{j_2} \eta_{j_4}] \EE[\eta_{j_1}] \EE[\eta_{j_3}] \notag \\[1.1mm]
        &\qquad\quad+ \EE[\eta_{j_3} \eta_{j_4}] \EE[\eta_{j_1}] \EE[\eta_{j_2}] - 3\, \EE[\eta_{j_1}] \EE[\eta_{j_2}] \EE[\eta_{j_3}] \EE[\eta_{j_4}] \notag \\[1.1mm]
        &\qquad= (r + 3)^{(4)} y_{j_1} y_{j_2} y_{j_3} y_{j_4} - (r + 2)^{(3)} y_{j_1} y_{j_2} y_{j_3} r y_{j_4} - (r + 2)^{(3)} y_{j_1} y_{j_2} y_{j_4} r y_{j_3} \notag \\
        &\qquad\quad- (r + 2)^{(3)} y_{j_1} y_{j_3} y_{j_4} r y_{j_2} - (r + 2)^{(3)} y_{j_2} y_{j_3} y_{j_4} r y_{j_1} + (r + 1)^{(2)} y_{j_1} y_{j_2} r y_{j_3} r y_{j_4} \notag \\
        &\qquad\quad+ (r + 1)^{(2)} y_{j_1} y_{j_3} r y_{j_2} r y_{j_4} + (r + 1)^{(2)} y_{j_1} y_{j_4} r y_{j_2} r y_{j_3} + (r + 1)^{(2)} y_{j_2} y_{j_3} r y_{j_1} r y_{j_4} \notag \\
        &\qquad\quad+ (r + 1)^{(2)} y_{j_2} y_{j_4} r y_{j_1} r y_{j_3} + (r + 1)^{(2)} y_{j_3} y_{j_4} r y_{j_1} r y_{j_2} - 3\, r^4 y_{j_1} y_{j_2} y_{j_3} y_{j_4} \\
        &\qquad= \boxed{3 r (2 + r) y_{j_1} y_{j_2} y_{j_3} y_{j_4}}.
    \end{align*}

\section{Open Problems}\label{sec:open.problems}

    Here are some research questions for the reader that are of interest:

    \begin{itemize}
        \item Using the moment formulas in the present paper, extend to the negative multinomial distribution the local limit theorem, total variation bound and Le Cam distance bound found in Lemma~1, Theorems~3 and~4 of \citet{doi:10.1007/s00184-023-00897-2} for the negative binomial distribution ($d = 1$).
        \item Using the moment formulas in the present paper, study the asymptotic properties of the Bernstein estimator with a negative multinomial kernel, as was carried out for the Bernstein estimator with a multinomial kernel on the simplex in \citet{MR4287788}.
        \item Investigate whether the negative multinomial distribution is a completely monotonic function of its parameters. This question was answered positively by \citet{MR3825458} and \citet{MR4201158} for the multinomial distribution, who showed that it is in fact even logarithmically completely monotonic. The same result was extended to a matrix-parametrized generalization by \citet{MR4460102}.
    \end{itemize}

%\section*{Acknowledgments}

    %We are grateful to the referee for reviewing the manuscript and providing valuable comments and suggestions that have led to enhancements in the writing of this paper.

\section*{Funding}

    F.\ Ouimet is supported financially by a postdoctoral fellowship (CRM-Simons) from the Centre de recherches math\'ematiques (Montr\'eal, Canada) and the Simons Foundation.

\section*{Data availability}

    Not applicable.

\section*{Conflicts of interest}

    The author declares no conflict of interest.

%
% ----------  B I B L I O G R A P H Y  ----------
%

\end{document}